\documentclass[11pt,twoside]{article}
\usepackage{amsmath}
\usepackage{amsfonts}
\usepackage{amsthm,amssymb}
\usepackage{graphicx}

\newtheorem{defi}{Definition}[section]
\newtheorem{teorema}[defi]{Theorem}
\newtheorem{lema}[defi]{Lemma}

\newtheorem{nota}[defi]{Remark}

\def\R {{\rm I}\hskip -0.85mm{\rm R}}

\def\a{\alpha}

\def\e{\varepsilon}
\def\D{\Delta}
\def\d{\delta}

\def\l{\lambda}

\def\O{\Omega}
\def\p{\partial}

\def\v{\varphi}

\def\ov{\overline}
\def\un{\underline}

\def\n{\nabla}

\begin{document}
\vskip 1cm
\begin{center}
\vskip 0.1cm
{\bf\huge Some remarks on the comparison principle in  Kirchhoff equations}
\end{center}
\vskip 0.3cm
\begin{center}

{\sc Giovany M. Figueiredo$^1$ and Antonio Su{\'a}rez$^2$,}
\\
\vspace{0.5cm}

1. Universidade Federal do Par\'a, Faculdade de Matem\'atica \\
CEP: 66075-110 Bel\'em - Pa , Brazil \\

2. Dpto. de Ecuaciones
Diferenciales y An{\'a}lisis Num{\'e}rico\\ Fac. de Matem{\'a}ticas, Univ. de Sevilla\\
C/. Tarfia s/n, 41012 - Sevilla, SPAIN,\\

\medskip

E-mail addresses: giovany@ufpa.br, suarez@us.es
 \end{center}
\vskip 1cm

\begin{abstract}
In this paper we study the validity of the comparison principle and the sub-supersolution method for Kirchhoff type equations. We show that these principles do not work when the Kirchhoff function is increasing, contradicting some previous results. We give an alternative sub-supersolution method and apply it to some models.
\end{abstract}

\vskip 1cm

\noindent{{\bf Key Words}. Kirchhoff equation, comparison principle, sub-supersolution method.}
\newline
\noindent{{\bf AMS Classification.} 45M20, 35J25, 34B18.} \vskip 0.4cm

\section{Introduction}

In the last years the nonlinear elliptic Kirchhoff equation  has attracted much attention, see for instance \cite{} and references therein. The equation has the following general form
\begin{equation} \label{uno}
\left \{ \begin{array}{ll}
-M(\|u\|^2)\D u=f(x,u) & \mbox{in $\O$,}\\
u=0 & \mbox{on $\p\O$,}
\end{array}\right.
\end{equation}
where $\O\subset\R^N$, $N\geq 1$, is a bounded and regular domain,
$$
 \|u\|^2:=\int_\O |\n u|^2dx,
$$
and $M$ is a continuous function verifying
$$
M:[0,+\infty)\mapsto[0,+\infty)\quad\mbox{and  $\exists m_0>0$ such that $M(t)\geq m_0>0\;\forall t\in\R$},
\leqno(M_0)
$$
and $f\in C(\ov\O\times\R)$. We assume $(M_0)$ along the paper.

To study this problem different methods have been used, mainly  variational methods and  fixed point arguments, and also bifurcation and sub-supersolution.

In this note, we have two main objectives. On one hand, we present some examples demonstrating the some comparison results appearing in the literature are not correct. On the other hand, we prove a sub-supersolution method that includes the above ones, remembered in this work.

An outline of the paper is as follows:
In Section 2 we recall the previous results related to comparison and sub-supersolution method, and we present our main result. In Section 3 we give  some counterexamples showing that some comparison results are not correct. Section 4 is devoted to prove our main result and in Section 5 we apply our result to some specific examples.

\setcounter{equation}{0}
\section{Previous and main results}

In our knowledge, there are basically  three results concerning  to the comparison and sub-supersolution results related to (\ref{uno}). Let us recall them. In \cite{julioclaudianor1} (Theorems 2 and 3) the following result was proved:
\begin{teorema}
\label{teo1}
Assume that:
\begin{enumerate}
\item[$(M_1)$]  $M$ is non-increasing in $[0,+\infty)$.
\item[$(H)$] Define the function,
$$
H(t):=M(t^2)t
$$
and assume that $H$ is increasing and $H(\R)=\R$.
\end{enumerate}
\begin{enumerate}
\item If there exist two non-negative functions $\un u,\ov u\in C^2(\ov\O)$ such that $\un u=\ov u=0$ on $\p\O$ and
\begin{equation}
\label{compara}
-M(\|\ov u\|^2)\D\ov u\geq -M(\|\un u\|^2)\D\un u\quad\mbox{in $\O$,}
\end{equation}
then (comparison principle)
$$
\un u\leq \ov u \quad\mbox{in $\ov\O$.}
$$
\item If
\begin{enumerate}
\item[$(f_1)$] $f$ is increasing in the variable $u$ for each $x\in\O$ fixed,
\end{enumerate}
and there exist two regular functions $0\leq \un u\leq \ov u$ in $\O$, $\un u=\ov u=0$ on $\p\O$ verifying
$$
-M(\|\ov u\|^2)\D\ov u\geq f(x,\ov u),\quad -M(\|\un u\|^2)\D\un u\leq f(x,\un u),\quad\mbox{in $\O$.}
$$
then (sub-supersolution method), there exists a solution $u$ of (\ref{uno}) such that $\un u\leq u\leq \ov u$ in $\O$.
\end{enumerate}
\end{teorema}

In \cite{mdg} (Theorems 3.2 and 3.3), se also \cite{chinos}, the case when $M$ is increasing was studied.  The authors proved a similar result to Theorem \ref{teo1}:
\begin{teorema}
\label{teo2}
Assume that:
\begin{enumerate}
\item[$(M_2)$]  $M$ is increasing.
\end{enumerate}
Then, the comparison principle holds. Moreover, if $f$ verifies $(f_1)$, the sub-supersolution method also works.
%If there exist two regular function such that
%$$
%-M(\|\ov u\|^2)\D\ov u\geq -M(\|\un u\|^2)\D\un u
%$$
%then (comparison principle)
%$$
%\un u\leq \ov u.
%$$
%Moreover, if
%\begin{enumerate}
%\item[$(H_f)$] $f$ is non-decreasing in $u$,
%\end{enumerate}
%and there exist two functions $\un u\leq \ov u$ in $\O$, $\un u=\ov u=0$ on $\p\O$ verifying
%$$
%-M(\|\ov u\|^2)\D\ov u\geq f(x,\ov u),\quad -M(\|\un u\|^2)\D\un u\leq f(x,\un u),
%$$
%then, there exists a solution $u$ of (\ref{uno}) such that $\un u\leq u\leq \ov u$ in $\O$.
\end{teorema}
Finally, in \cite{julioclaudianor2} the following result is shown:
\begin{teorema}
\label{teo3}
Assume that $M$ verifies $(M_2)$ and
\begin{enumerate}
%\item[$(M_2)$]  $M$ is increasing.
\item[$(f_2)$] $f$ is a positive function.
\end{enumerate}
If there exist  $\ov u\in W^{1,\infty}(\Omega)$, $\ov u\geq 0$ on $\p\O$, and a family $(\un u_\delta)\subset W_0^{1,\infty}(\Omega)$ such that
$$
-m_0\D\ov u\geq f(x,\ov u),
$$
$\|\un u_\delta\|_{1,\infty}\to 0$ as $\delta\to 0$, $\un u_\d\leq\ov u$ in $\O$ for $\d$ small enough, and given $\a>0$, there is $\d_0$ such that
$$
-\D\un u_\d\leq \frac{1}{\a}f(x,\un u_\d),\quad\mbox{for $\d\leq \d_0$,}
$$
then, there is a small enough $\d>0$ such that there exists a solution $u$ of (\ref{uno}) such that $\un u_\d\leq u\leq \ov u$ in $\O$.
\end{teorema}
Of course, the above inequalities are considered in the weak sense.

Our main result reads as follows:
\begin{teorema}
\label{teo4}
Assume that:
\begin{enumerate}
\item[$(M_3)$]  $G(t):=M(t)t$ is invertible and denote by $R(t):=G^{-1}(t).$
\end{enumerate}
Define now the non-local operator
$$
{\cal R}(w):=R\left(\int_\O f(x,w)wdx\right).
$$
If there exist $\un u,\ov u\in H^1(\O)\cap L^\infty(\O)$ such that  $\un u\leq \ov u$ in $\O$, $\un u \leq \ov u$ on $\p\O$ verifying
$$
-M({\cal R}(w))\D\ov u\geq f(x,\ov u),\quad -M({\cal R}(w))\D\un u\leq f(x,\un u),\quad\forall w \in [\un u,\ov u],
$$
then, there exists a solution $u$ of (\ref{uno}) such that $\un u\leq u\leq \ov u$ in $\O$.
\end{teorema}

In Section 3, we show that Theorem \ref{teo2} is not correct.  Now, we deduce parts b) of Theorems \ref{teo1} and \ref{teo3} from Theorem \ref{teo4}.

\noindent\underline{Theorem \ref{teo4} implies Theorem \ref{teo1} b):} Assume that we have the hypotheses of Theorem \ref{teo1}. Observe that if $M$ is non-increasing and $H$ increasing, and assuming regularity on the functions,  then $G$ is increasing. Indeed, using that $M'\leq 0$ we get that
$$
G'(t)=M'(t)t+M(t)\geq 2tM'(t)+M(t)=H'(t^{1/2})>0,\quad t>0,
$$
and then, $w\mapsto {\cal R}(w) $ is increasing because $f$ is also increasing.

Consider now that $\un u$, $\ov u$ is sub-supersolution in the sense of Theorem \ref{teo1}, we are going to show that they are also sub-super in the sense of Theorem \ref{teo4}. We show this fact with $\un u$, for $\ov u$ we can apply an analogous reasoning. Since $f$ is increasing, we have that
$$
M(\|\un u\|^2)\|\un u\|^2\leq \int_\O f(x,\un u)\un u\leq \int_\O f(x,w)w\quad\forall w\in [\un u,\ov u],
$$
and so,
$$
\|\un u\|^2\leq {\cal R}(w)\Longrightarrow M(\|\un u\|^2)\geq M({\cal R}(w))\quad\forall w\in [\un u,\ov u].
$$
Hence,
$$
-\D \un u\leq \frac{f(x,\un u)}{M(\|\un u\|^2)}\leq  \frac{f(x,\un u)}{M({\cal R}(w))}\quad\forall w\in [\un u,\ov u].
$$
Then, $(\un u,\ov u)$ verifies the hypotheses of Theorem \ref{teo4} and we can conclude the existence of a solution $u$ of (\ref{uno}) and $u\in [\un u,\ov u]$.

\noindent\underline{Theorem \ref{teo4} implies Theorem \ref{teo3} b):} Indeed, assume that we have the hypotheses of Theorem \ref{teo3}. Observe that if $M$ is increasing,  then $G$ is increasing. Assume now the existence of a supersolution  $\ov u$ and family of sub-solution $\un u_\d$ in the sense of Theorem \ref{teo3}. Then,
$$
-\D \ov u\geq\frac{f(x,\ov u)}{m_0}\geq \frac{f(x,\ov u)}{M({\cal R}(w))}\quad\forall w\in [\un u,\ov u].
$$

Consider now
$$
\a=\max_{0\leq w\leq \ov u}M({\cal R}(w)),
$$
and take $\un u=\un u_\d$ for some $\d\leq \d_0$ given by Theorem \ref{teo3}. Then, using that $f\geq 0$,
$$
-\D \un u\leq \frac{1}{\a}f(x,\un u)\leq \frac{1}{M({\cal R}(w))}f(x,\un u)\quad\forall w\in [\un u,\ov u],
$$
and so,  $\un u$, $\ov u$ is sub-supersolution in the sense of Theorem \ref{teo4}.

\setcounter{equation}{0}
\section{Counterexamples}
In this section we have two objectives: when $M$ verifies $(M_2)$  or $M$ verifies $(M_1)$ and not $(M_3)$, the comparison principle fails.

For that consider
$$
\O=(0,\pi),\quad \un u:=\sin(x),\quad\ov u:=\rho x(\pi-x),\; \rho>0.
$$
Observe that
$$
\max_{x\in[0,\pi]}\frac{\un u(x)}{\ov u(x)}=\max_{x\in[0,\pi]}\frac{\sin(x)}{x(\pi-x)}=\frac{4}{\pi^2}:=\rho^*\simeq 0.4083
$$
Hence, for $\rho<\rho^*$ we have that $\un u\nleq \ov u$ in $\O$. On the other hand
$$
\|\un u\|^2=\frac{\pi}{2},\qquad\|\ov u\|^2=\rho^2\frac{\pi^3}{3}.
$$
So,
$$
-M(\|\un u\|^2)\D\un u\leq -M(\|\ov u\|^2)\D\ov u\quad\mbox{for all $x\in (0,\pi)$},
$$
 if and only if
 \begin{equation}
 \label{condi}
 M(\pi/2)\leq 2\rho M\left(\rho^2\frac{\pi^3}{3}\right).
 \end{equation}

Consider
$$
M(t):=a+b(t+c)^p,\quad a,c\geq 0, b>0,p\in \R.
$$
\underline{Case 1: $M$ is increasing.} Consider in this case $c=0$, $a>0$ and $p>0$. Then, (\ref{condi}) is equivalent to
$$
a+b\left(\frac{\pi}{2}\right)^p\leq 2\rho \left(a+b\left(\rho^2\frac{\pi^3}{3}\right)^p\right).
$$
Taking $b$ large, we need that for some $\rho<\rho^*$
$$
\frac{1}{2}\frac{1}{\left(\pi^2\frac{2}{3}\right)^p}<\rho^{1+2p}.
$$
By continuity, it is enough that the above inequality holds for $\rho=\rho^*$, that is,
$$
\frac{1}{2}\frac{1}{\left(\pi^2\frac{2}{3}\right)^p}<\left(\frac{4}{\pi^2}\right)^{1+2p}\Longleftrightarrow 1<\frac{8}{\pi^2}\left(\frac{32}{3\pi^2}\right)^p,
$$
which is true for $p$ large.

\noindent\underline{Case 2: $M$ is decreasing.} Take in this case $a,c>0$ and $p<0$. In this case, we need that
$$
a(1-2\rho)\leq b\left(2\rho\left(\rho^2\frac{\pi^3}{3}+c\right)^p-\left(\frac{\pi}{2}+c\right)^p\right),
$$
for that we need, taking $b$ large, that
$$
2\rho\left(\rho^2\frac{\pi^3}{3}+c\right)^p>\left(\frac{\pi}{2}+c\right)^p\Longleftrightarrow 2\rho>\left(\frac{\frac{\pi}{2}+c}{\rho^2\frac{\pi^3}{3}+c}\right)^p.
$$
Take $\rho$ small such that
$$
\frac{\frac{\pi}{2}+c}{\rho^2\frac{\pi^3}{3}+c}>1.
$$
Now, we have the above inequality taking $p$ very negative.
\begin{nota}
The above example shows that Theorem  3.3 in \cite{mdg},  see also Theorem 2.3 in \cite{hd}, seems not correct. Hence, some other papers where these results have been applied are also not correct, see for instance, \cite{lei}, \cite{liu}, \cite{chen}, \cite{sun}, \cite{afrouzi}, \cite{ngu}.
\end{nota}
\setcounter{equation}{0}
\section{Proof of Theorem \ref{teo4}:}
First, we are going to transforms our equation (\ref{uno})  into another non-local elliptic equation. Indeed,
multiplying (\ref{uno}) by $u$ and integrating, we get
$$
M(\|u\|^2)\|u\|^2=\int_\O f(x,u)udx.
$$
By $(H_3)$, $G$ is invertible, and so
$$
\|u\|^2=R\left(\int_\O f(x,u)u dx\right)={\cal R}(u).
$$
Then, (\ref{uno}) is equivalent to problem
\begin{equation} \label{dos}
\left \{ \begin{array}{ll}
-\D u=\displaystyle\frac{f(x,u)}{M({\cal R}(u))} & \mbox{in $\O$,}\\
u=0 & \mbox{on $\p\O$.}
\end{array}\right.
\end{equation}
Observe that (\ref{dos}) is a non-local elliptic equation, without terms in $\|u\|$, and so it suffices to apply Theorem 3.2 in \cite{julionosotrosADE}. This completes the proof.

In the following result, we prove a specific comparison principle which is valid when $M$ verifies only hypothesis $(H)$. Define $e$ the unique positive solution of the equation
\begin{equation} \label{e}
\left \{ \begin{array}{ll}
-\D e=1 & \mbox{in $\O$,}\\
e=0 & \mbox{on $\p\O$.}
\end{array}\right.
\end{equation}

\begin{lema}
\label{espe}
Assume that $M$ verifies $(H)$ and let $u_i$, $i=1,2$  functions such that
$$
-M(\|u_i\|^2)\D u_i=f_i\in\R_+
$$
and $f_1\leq f_2$. Then, $u_1\leq u_2$ in $\O$.
\end{lema}
\begin{proof}
Observe that
\begin{equation}
\label{estrella}
M(\|u_i\|^2)u_i=f_ie,
\end{equation}
and then, $u_1\leq u_2$ if and only if
\begin{equation}
\label{equi1}
\frac{f_1}{M(\|u_1\|^2)}\leq \frac{f_2}{M(\|u_2\|^2)}.
\end{equation}
But observe that form (\ref{estrella})
$$
f_i=\frac{M(\|u_i\|^2)\|u_i\|}{\|e\|},
$$
and then (\ref{equi1}) is equivalent to
\begin{equation}
\label{equi2}
\|u_1\|\leq \|u_2\|.
\end{equation}
Since $M(\|u_i\|^2)\|u_i\|=f_i\|e\|$ and due to $(H)$, it follows (\ref{equi2})
\end{proof}

\setcounter{equation}{0}
\section{Applications}
In this section we apply our result to some models. We only  assume that $M$ verifies  $(M_3)$.  Denote by $\l_1>0$ the principal eigenvalue of the Laplacian and $\v_1>0$ the eigenfunction associated to it with $\|\v_1\|_\infty=1$.

\noindent\underline{Example 1:} Consider the equation
\begin{equation} \label{ex1}
\left \{ \begin{array}{ll}
-M(\|u\|^2)\D u=\l u^q & \mbox{in $\O$,}\\
u=0 & \mbox{on $\p\O$,}
\end{array}\right.
\end{equation}
where $\l\in\R$ and $0<q<1$.  This problem was analyzed in \cite{julioclaudianor1} when $M$ verifies $(M_1)$ and $(H)$. We are going to show that (\ref{ex1}) possesses a positive solution if and only if $\l>0$. From the maximum principle, if $\l\leq 0$ problem (\ref{ex1}) does not have any positive solution. Assume $\l>0$ and
take as sub-supersolutions $\un u=\e\v_1$ and $\ov u=Ke$ with $\e,K>0$ to be chosen.
Then, $\ov u$ is supersolution if
$$
K^{1-q}\geq \frac{1}{m_0}\lambda \|e\|_\infty^q.
$$
Fix such $K$. Then, $\un u$ is subsolution if
$$
M({\cal R}(w))\e^{1-q}\leq\frac{\lambda}{\lambda_1},\quad\forall w\in[\un u,\ov u].
$$
It is enough to take $\e$ small such that the above inequality holds and that $\un u\leq \ov u$.

\noindent\underline{Example 2:} Consider the classical concave-convex equation
\begin{equation} \label{ex2}
\left \{ \begin{array}{ll}
-M(\|u\|^2)\D u=\l u^q +u^p& \mbox{in $\O$,}\\
u=0 & \mbox{on $\p\O$,}
\end{array}\right.
\end{equation}
where $\l\in\R$ and $0<q<1<p$. Again we assume only that $M$ verifies $(M_3)$. We show that there exists at least a positive solution for $\l$ small and positive. For that, again take the same sub-supersolution of the above example. We can show that $\ov u=Ke$ es supersolution provided of
$$
m_0K^{1-q}\geq \l\|e\|_\infty^q+K^{p-q}\|e\|_\infty^p.
$$
Then, there exists $\l_0>0$ such that for $\l\in (0,\l_0)$, there exists $K_0$ such that $\ov u=K_0e $ is supersolution.

Now, $\un u=\e\v_1$ is subsolution provided of
$$
M({\cal R}(w))\e^{1-q}\lambda_1\leq \lambda+\e^{p-q}\v_1^{p-q},\quad\forall w\in[\un u,\ov u].
$$
It suffices again to take $\e$ small.

\noindent\underline{Example 3:} Consider now the logistic  equation
\begin{equation} \label{ex3}
\left \{ \begin{array}{ll}
-M(\|u\|^2)\D u=\l u -u^p& \mbox{in $\O$,}\\
u=0 & \mbox{on $\p\O$,}
\end{array}\right.
\end{equation}
where $\l\in\R$ and $1<p$. This equation was studied in \cite{chipotjulio} when $M=M(u)$ is a continuos function from $L^p(\O)$ into $\R$, and $m_0\leq M\leq m_\infty$.  They used a fixed point argument  and showed the existence of positive solution for $\l>\l_1m_\infty$ (Theorem 2.1).

We obtain a similar result for the Kirchhoff equation (\ref{ex3}) by the sub-supersolution method.
Take $\ov u=\l$, it is clear that $\ov u$ is supersolution. As subsolution $\un u=\e\v_1$. Then, we need that
$$
M({\cal R}(w) )\l_1+(\e\v_1)^{p-1}\leq \l, \quad\forall w\in[\un u,\ov u].
$$
Then, there exists at least a positive solution for $\l>\l_1m_\infty.$

\vskip0.5cm \noindent {\bf Acknowledgements.}  AS by Ministerio de Econom\'ia y Competitividad under grant MTM2012-31304.

\end{document}